\documentclass[draftclsnofoot,onecolumn]{IEEEtran}
\usepackage{def}
\begin{document}

\title{Dual Acceleration for Minimax Optimization: Linear Convergence Under Relaxed Assumptions}

\author{Jingwang Li and Xiao Li, \IEEEmembership{Member, IEEE}
  \thanks{Jingwang Li was with the School of Data Science, The Chinese University of Hong Kong, Shenzhen 518172, China. He is now with the Department of Electronic and Computer Engineering, The Hong Kong University of Science and Technology, Hong Kong 99077, China (e-mail: jingwang.li@connect.ust.hk).}
  \thanks{Xiao Li is with the School of Data Science, The Chinese University of Hong Kong, Shenzhen 518172, China (e-mail: lixiao@cuhk.edu.cn).}
  \thanks{\textit{Corresponding author: Xiao Li.}}
}

\maketitle

\begin{abstract}
  This paper addresses the bilinearly coupled minimax optimization problem: $\min_{x \in \mathbb{R}^{d_x}}\max_{y \in \mathbb{R}^{d_y}} \ f_1(x) + f_2(x) + y^{\top} Bx - g_1(y) - g_2(y)$, where $f_1$ and $g_1$ are smooth convex functions, $f_2$ and $g_2$ are potentially nonsmooth convex functions, and $B$ is a coupling matrix. Existing algorithms for solving this problem achieve linear convergence only under stronger conditions, which may not be met in many scenarios. We first introduce the Primal-Dual Proximal Gradient (PDPG) method and demonstrate that it converges linearly under an assumption where existing algorithms fail to achieve linear convergence. Building on insights gained from analyzing the convergence conditions of existing algorithms and PDPG, we further propose the inexact Dual Accelerated Proximal Gradient (iDAPG) method. This method achieves linear convergence under weaker conditions than those required by existing approaches. Moreover, even in cases where existing methods guarantee linear convergence, iDAPG can still provide superior theoretical performance in certain scenarios.
\end{abstract}

\begin{IEEEkeywords}
  Minimax optimization, accelerated algorithms, inexact methods, linear convergence.
\end{IEEEkeywords}

\IEEEpeerreviewmaketitle

\section{Introduction}

\begin{table}[h]
  \caption{The oracle complexities of SOTA first-order algorithms to solve different cases of \cref{main_pro}, along with the corresponding lower bounds (if available).}
  \label{1229}
  \scriptsize
  \renewcommand{\arraystretch}{0.5}
  \centering
    \begin{threeparttable}[b]
      \begin{tabularx}{\textwidth}{ccC}
        \toprule
                                                             & Additional assumptions & Oracle complexity\tnote{1}                                                                                                                                                                                        \\
        \midrule
        \multicolumn{3}{c}{\textbf{Strongly-Convex-Strongly-Concave Case: \cref{basic_assump}, $g_1$ is $\mu_y$-strongly convex}}                                                                                                                                                                                 \\
        \midrule
        LPD \cite{thekumparampil2022lifted}, ABPD-PGS \cite{luo2024accelerated}                  &                        & $\bO{\max\pa{\sqrt{\kappa_x}, \sqrt{\kappa_{xy}}, \sqrt{\kappa_y}} \log \pa{\frac{1}{\epsilon}}}$                                                                                 \\
        \midrule
        APDG \cite{kovalev2022accelerated}                   &  $f_2 = 0$, $g_2 = 0$  & $\bO{\max\pa{\sqrt{\kappa_x}, \sqrt{\kappa_{xy}}, \sqrt{\kappa_y}} \log \pa{\frac{1}{\epsilon}}}$ \\
        \midrule
        \rowcolor{bgcolor}
        iDAPG                                          &                        & \makecell{$\mathcal{A}$: $\tilde{\mO}\pa{\sqrt{\kappa_x}\max\pa{\sqrt{\kappa_{xy}}, \sqrt{\kappa_y}} \log \pa{\frac{1}{\epsilon}}}$ \tnote{2}                                               \\ $\mathcal{B}$: $\bO{\max\pa{\sqrt{\kappa_{xy}}, \sqrt{\kappa_y}} \log \pa{\frac{1}{\epsilon}}}$}                        \\
        \midrule
        Lower bound\tnote{3} \cite{zhang2022lower}                    & $f_2 = 0$, $g_2 = 0$   & $\low{\max\pa{\sqrt{\kappa_x}, \sqrt{\kappa_{xy}}, \sqrt{\kappa_y}} \log \pa{\frac{1}{\epsilon}}}$                                                                                      \\
        \midrule
        \multicolumn{3}{c}{\textbf{Strongly-Convex-Full-Rank Case: \cref{basic_assump}, $f_2 = 0$, $B$ has full row rank}}                                                                                                                                                                                        \\
        \midrule
        APDG \cite{kovalev2022accelerated}                   & $g_2 = 0$              & $\bO{\max\pa{\sqrt{\kappa_{xy'}}, \sqrt{\kappa_x\kappa_B}, \kappa_B} \log \pa{\frac{1}{\epsilon}}}$                                                         \\
        \midrule
        \rowcolor{bgcolor}
        iDAPG                                          &                        & \makecell{$\mathcal{A}$: $\tilde{\mO}\pa{\sqrt{\kappa_x}\max\pa{\sqrt{\kappa_{xy'}}, \sqrt{\kappa_x\kappa_B}} \log \pa{\frac{1}{\epsilon}}}$                          \\ $\mathcal{B}$: $\bO{\max\pa{\sqrt{\kappa_{xy'}}, \sqrt{\kappa_x\kappa_B}} \log \pa{\frac{1}{\epsilon}}}$}                        \\
        \midrule
        \multicolumn{3}{c}{\textbf{Strongly-Convex-Linear Case: \cref{basic_assump}, $f_2 = 0$, $g_2 = 0$, $g_1$ is linear}}                                                                                                                                                                           \\
        \midrule
        Algorithm 1 \cite{salim2022optimal}               &                        & \makecell{$\mathcal{A}$: $\bO{\sqrt{\kappa_x} \log \pa{\frac{1}{\epsilon}}}$, $\mathcal{B}$: $\bO{\sqrt{\kappa_x\kappa_{B'}} \log \pa{\frac{1}{\epsilon}}}$}                          \\
        \midrule
        APDG \cite{kovalev2022accelerated}                   &                        & $\bO{\sqrt{\kappa_x\kappa_{B'}} \log \pa{\frac{1}{\epsilon}}}$
        \\
        \midrule
        \rowcolor{bgcolor}
        iDAPG                                          &                        & \makecell{$\mathcal{A}$: $\tilde{\mO}\pa{\kappa_x\sqrt{\kappa_{B'}} \log \pa{\frac{1}{\epsilon}}}$, $\mathcal{B}$: $\bO{\sqrt{\kappa_x\kappa_{B'}} \log \pa{\frac{1}{\epsilon}}}$} \\
        \midrule
        Lower bound \cite{salim2022optimal}                  &                        & $\mathcal{A}$: $\low{\sqrt{\kappa_x} \log \pa{\frac{1}{\epsilon}}}$, $\mathcal{B}$: $\low{\sqrt{\kappa_x\kappa_{B'}} \log \pa{\frac{1}{\epsilon}}}$                                   \\
        \midrule
        \multicolumn{3}{c}{\textbf{The Case that Satisfies \cref{basic_assump,main_pro_assumption}}}                                                                                                                                                                                \\
        \midrule
        \rowcolor{bgcolor}
        PDPG        & $g_3 = 0$              & $\bO{\max\pa{\kappa_{xy^2}, \kappa_x\kappa_{xy^3}} \log \pa{\frac{1}{\epsilon}}}$                                                                           \\
        \midrule
        \rowcolor{bgcolor}
        iDAPG &                        & \makecell{$\mathcal{A}$: $\tilde{\mO}\pa{\sqrt{\kappa_x}\max\pa{\sqrt{\kappa_{xy^2}}, \sqrt{\kappa_x\kappa_{xy^3}}} \log \pa{\frac{1}{\epsilon}}}$ \\ $\mathcal{B}$: $\bO{\max\pa{\sqrt{\kappa_{xy^2}}, \sqrt{\kappa_x\kappa_{xy^3}}} \log \pa{\frac{1}{\epsilon}}}$}                        \\
        \midrule
        \multicolumn{3}{c}{\textbf{Dual-Strongly-Convex Case: \cref{basic_assump}, $\varphi$ is $\mu_{\varphi}$-strongly convex}}                                                                                                                                     \\
        \midrule
        \rowcolor{bgcolor}
        iDAPG &                        & \makecell{$\mathcal{A}$: $\tilde{\mO}\pa{\sqrt{\kappa_x}\max\pa{\sqrt{\frac{L_y}{\mu_{\varphi}}}, \frac{\os(B)}{\sqrt{\mu_x\mu_{\varphi}}}} \log \pa{\frac{1}{\epsilon}}}$                            \\ $\mathcal{B}$: $\bO{\max\pa{\sqrt{\frac{L_y}{\mu_{\varphi}}}, \frac{\os(B)}{\sqrt{\mu_x\mu_{\varphi}}}} \log \pa{\frac{1}{\epsilon}}}$}                        \\        
        \bottomrule
      \end{tabularx}
      \begin{tablenotes}
        \item $\kappa_x = \frac{L_x}{\mu_x}$, $\kappa_y = \frac{L_y}{\mu_y}$, $\kappa_{B} = \frac{\os^2(B)}{\mins^2(B)}$, $\kappa_{B'} = \frac{\os^2(B)}{\us^2(B)}$, $\kappa_{xy} = \frac{\os^2(B)}{\mu_x\mu_y}$, $\kappa_{xy'} = \frac{L_xL_y}{\mins^2(B)}$, $\kappa_{xy^2} = \frac{L_xL_y}{\mine{BB\T + L_xP}}$, $\kappa_{xy^3} = \frac{\os^2(B)}{\mine{BB\T + L_xP}}$.
        \item [1] $\mathcal{A}$: $\nabla f_1$ and $\prox_{f_2}$; $\mathcal{B}$: $B$, $B\T$, $\nabla g_1$, and $\prox_{g_2}$. If only one complexity is provided, it suggests that the oracle complexities of $\mathcal{A}$ and $\mathcal{B}$ are the same.
        \item [2] $\tilde{\mO}$ hides a logarithmic factor that depends on the problem parameters; refer to \cref{iDAPG_complexity} for further details.
        \item [3] When only one lower bound is provided, it suggests that only the lower bound of the maximum of the oracle complexities of $\mathcal{A}$ and $\mathcal{B}$ is available.
      \end{tablenotes}
    \end{threeparttable}
\end{table}

In this paper, we consider the following minimax optimization problem:
\begin{equation} \label{main_pro} \tag{P1}
  \min_{x \in \R^{d_x}}\max_{y \in \R^{d_y}} \ \cL(x, y) = f_1(x) + f_2(x) + y\T Bx - g_1(y) - g_2(y),
\end{equation}
where $f_1:\R^{d_x} \rightarrow \R$ and $g_1:\R^{d_y} \rightarrow \R$ are smooth and convex functions, $f_2:\R^{d_x} \rightarrow \exs$ and $g_2:\R^{d_y} \rightarrow \exs$ are convex but possibly nonsmooth functions, and $B \in \R^{d_y \times d_x}$ is a coupling matrix. Note that a solution to \cref{main_pro} corresponds to a saddle point of $\cL$. Without loss of generality, we assume that there exists at least one solution $(x^*, y^*)$ to \cref{main_pro}.

The general minimax formulation in \cref{main_pro} finds broad applications across machine learning and optimization, including robust optimization \cite{ben2009robust}, reinforcement learning \cite{du2019linear}, supervised learning \cite{kovalev2022accelerated}, and so on. To illustrate this versatility, consider empirical risk minimization (ERM) with generalized linear models (GLMs) \cite{hardin2018generalized}. Given a dataset $X \in \R^{p \times d}$ with $p$ samples and $d$ features, the ERM problem can be formulated as
\eqe{ \label{1255}
  \min_{\theta \in \R^d} f(\theta) + g(\theta) + \ell(X\theta),
}
where $\theta \in \R^d$ is the model parameter, $\ell: \R^p \rightarrow \exs$ is a potentially nonsmooth convex loss function, $f: \R^d \rightarrow \R$ is a convex and smooth function (e.g., $\ell_2$ regularization), and $g: \R^d \rightarrow \exs$ is a potentially nonsmooth convex function (e.g., $\ell_1$ regularization or indicator functions). In the case of linear regression, the loss function is given by $\ell(z) = \frac{1}{2p}\norm{z-y}^2$; for logistic regression, it is expressed as $\ell(z) = \frac{1}{p}\sum_{j=1}^p\log\pa{1+e^{-y_jz_j}}$ ($y_j \in \set{1,-1}$). Instead of solving \cref{1255} directly, we can also address its equivalent minimax problem:
\eqe{ \label{1256}
  \min_{\theta \in \R^d} \max_{\lambda \in \R^p} f(\theta) + g(\theta) + \lambda\T X\theta - \ell^*(\lambda),
}
which is clearly a special case of \cref{main_pro}. This minimax formulation is advantageous in many scenarios, such as when it allows for a finite-sum structure \cite{wang2017exploiting} or introduces a sparsity structure \cite{lei2017doubly}.

In this work, we focus on designing accelerated first-order algorithms to solve \cref{main_pro} and achieve linear convergence, under the assumption that at least one of $f_1$, $f_2$, $g_1$, and $g_2$ is strongly convex. The primal-dual hybrid gradient (PDHG) method, one of most popular first-order algorithms for solving \cref{main_pro}, has established linear convergence as early as \cite{chambolle2011first, chambolle2016ergodic}, provided that both $f_2$ and $g_2$ are strongly convex. Under a similar condition where both $f_1$ and $g_1$ are strongly convex, several accelerated algorithms have been proposed that demonstrate faster linear convergence rates \cite{thekumparampil2022lifted, kovalev2022accelerated, luo2024accelerated}.

The first attempt to relax the strong convexity condition was made in \cite{duchi2011dual}, where the linear convergence of PDHG is established by replacing the strong convexity of $g_1$ with the condition that $f_2=0$ and $A$ has full row rank. Under this condition, several accelerated algorithms have been proposed in \cite{zhang2022near, kovalev2022accelerated}, leading to improved linear convergence rates. Another case of \cref{main_pro} that allows existing algorithms to achieve linear convergence without requiring both primal and dual strong convexity is the linearly constrained optimization problem: $\min_{x \in \R^d} f(x) \ \st \ A x = b$, where $f$ is smooth and strongly convex. This problem is special case of \cref{main_pro} with $f_1 = f$, $f_2 = 0$, $g_2 = 0$, and $g_1(y) = b\T y$. For this linearly constrained optimization problem, the algorithm proposed in \cite{salim2022optimal} has been shown to achieve linear convergence and match the lower complexity bound.

In summary, when either $f_1$ or $f_2$ is strongly convex, existing algorithms for solving \cref{main_pro} achieve linear convergence only when at least one of the following additional conditions is met:
\begin{enumerate}
  \item $g_1$ or $g_2$ is strongly convex (PDHG \cite{chambolle2011first,chambolle2016ergodic}, LPD \cite{thekumparampil2022lifted}, ABPD-PGS \cite{luo2024accelerated}, APDG \cite{kovalev2022accelerated});
  \item $f_2 = 0$ and $g_2 = 0$, and $B$ has full row rank (PDHG \cite{du2019linear}, APDG);
  \item $f_2 = 0$ and $g_2 = 0$, and $g_1$ is a linear function (Algorithm 1 of \cite{salim2022optimal}, APDG).
\end{enumerate}
To the best of our knowledge, no existing algorithm achieves linear convergence under conditions weaker than those listed above. However, certain real-world problems may not satisfy these conditions, raising an important question:

Can we design algorithms capable of solving \cref{main_pro} and achieving linear convergence under weaker conditions than those currently required?

\textbf{This work provides a definitive resolution to the aforementioned question through two key contributions:}
\begin{enumerate}
  \item We first propose PDPG, an extension of Algorithm 1 introduced in \cite{chambolle2016ergodic}. We prove that PDPG converges linearly under \cref{main_pro_assumption}, whereas existing methods fail to guarantee linear convergence under this assumption.
  \item Building on insights gained from analyzing the convergence conditions of existing algorithms and PDPG, we further propose iDAPG, which achieves linear convergence under weaker conditions than those required by existing algorithms. Notably, even in cases where existing methods guarantee linear convergence, iDAPG can still provide superior theoretical performance in certain scenarios; see \cref{1229} for detailed comparisons.
\end{enumerate}

\textit{Notations:} We use the standard inner product $\dotprod{\cdot, \cdot}$ and the standard Euclidean norm $\norm{\cdot}$ for vectors, along with the standard spectral norm $\norm{\cdot}$ for matrices.
For a symmetric matrix $A \in \R^{n \times n}$, let $\mine{A}$ and $\ove{A}$ denote the smallest and largest eigenvalues of $A$, respectively. We denote $A > 0$ (or $A \geq 0$) to indicate that $A$ is positive definite (or positive semi-definite). For a matrix $B \in \R^{m \times n}$, let $\mins(B)$, $\us(B)$, and $\os(B)$ represent the smallest singular value, the smallest nonzero singular value, and the largest singular values of $B$, respectively.
For a function $f: \R^n \rightarrow \exs$, $S_f(x)$ denotes one of its subgradients at $x$, $\partial f(x)$ denotes its subdifferential at $x$. The proximal operator of $f$ is given by $\prox_{\alpha f}(x) = \arg\min_{y} f(y) + \frac{1}{2\alpha}\norm{y-x}^2$ for $\alpha > 0$. Additionally, the Fenchel conjugate of $f$ is defined as $f^*(y) = \sup_{x \in \R^n}y\T x-f(x)$.

\section{Linear Convergence of PDPG under \cref{basic_assump,main_pro_assumption}}

\begin{algorithm}[t]
  \caption{Primal-Dual Proximal Gradient Method (PDPG)}
  \label{alg:PDPG}
  \begin{algorithmic}[1]
    \Require
    $T>0$, $\alpha>0$, $\beta>0$, $\theta \geq 0$, $x^0$, $y^0$
    \Ensure
    $x^T$, $y^T$
    \For {$k=0,\dots, T-1$ }
    \State $x\+ = \prox_{\alpha f_2}\pa{x^k - \alpha\pa{\nabla f_1(x^k) + B\T y^k}}$
    \State \eqe{y\+ = &\prox_{\beta g_2}\Big(y^k \\&- \beta\pa{\nabla g_1(y^k) - B\pa{x\+ + \theta(x\+-x^k)}}\Big) \nonumber}
    \EndFor
  \end{algorithmic}
\end{algorithm}

\begin{algorithm}[t]
  \caption{Inexact Dual Accelerated Proximal Gradient Method (iDAPG)}
  \label{alg:iDAPG}
  \begin{algorithmic}[1]
    \Require
    $T>0$, $L_{\varphi}$, $\mu_{\varphi}$, $x^0$, $y^0$
    \Ensure
    $x^T$, $y^T$
    \State $z^0 = y^0$
    \State Set $\beta_k = \frac{\sqrt{\kappa_{\varphi}}-1}{\sqrt{\kappa_{\varphi}}+1}$ if $\mu_{\varphi} > 0$, where $\kappa_{\varphi} = \frac{L_{\varphi}}{\mu_{\varphi}}$; otherwise set $\beta_k = \frac{k}{k+3}$.
    \For {$k=0,\dots, T-1$ }
    \State Solve
    \eqe{ \label{1356}
      \min_{x \in \R^{d_x}} f_1(x) + f_2(x) + \dotprod{B\T z^k, x}
    }
    to obtain an inexact solution $x\+$.
    \State $y\+ = \prox_{\frac{1}{L_{\varphi}} g_2}\pa{z^k - \frac{1}{L_{\varphi}}\lt(\nabla g_1(z^k) - Bx\+\rt)}$
    \State $z\+ = y\+ + \beta_k\lt(y\+-y^k\rt)$
    \EndFor
  \end{algorithmic}
\end{algorithm}

Throughout this paper, we assume that the following assumption holds:
\begin{assumption} \label{basic_assump}
  $f_1$, $f_2$, $g_1$, and $g_2$ satisfy
  \begin{enumerate}
    \item $f_1$ is $\mu_x$-strongly convex and $L_x$-smooth with $L_x \geq \mu_x > 0$;
    \item $g_1$ is convex and $L_y$-smooth with $L_y \geq 0$;
    \item $f_2$ and $g_2$ are proper\footnote{We say a function $f$ is proper if $f(x)>-\infty$ for all $x$ and $\dom{f} \neq \emptyset$.} convex, lower semicontinuous and proximal-friendly\footnote{We say a function $f$ is proximal-friendly if $\prox_{\alpha f}(x)$ can be easily computed for any $x$.}.
  \end{enumerate}
\end{assumption}

To address the aforementioned question, we begin with the following assumption:
\begin{assumption} \label{main_pro_assumption}
  $f_2 = 0$, $g_1(y) = g_3(y) + \frac{1}{2}y\T Py + y\T b$, where $g_3$ is a smooth convex function and $P \geq 0$. Furthermore, $BB\T + cP > 0$ for any $c > 0$.
\end{assumption}
It is evident that none of the aforementioned linear convergence conditions is satisfied under \cref{main_pro_assumption}. Nevertheless, we demonstrate that PDPG, an extension of Algorithm 1 introduced in \cite{chambolle2016ergodic}, achieves linear convergence under \cref{main_pro_assumption}.

\begin{theorem} \label{PDPG_convergence}
  Assume that \cref{basic_assump,main_pro_assumption} holds, $g_3 = 0$, $\theta = 0$, and
  \eqe{ \label{1428}
    \alpha < \frac{1}{L_x}, \ \beta \leq \frac{\mu_x}{\os^2(B)+\mu_x\ove{P}}.
  }
  Then, $x^k$ and $y^k$ generated by PDPG satisfy
  \eqe{
    &c_x\norm{x^k-x^*}^2 + c_y\norm{y^k-y^*}^2 \\
    \leq &\delta^k\pa{c_x\norm{x^0-x^*}^2 + c_y\norm{y^0-y^*}^2}, \ \forall k \geq 0,
  }
  where $(x^*, y^*)$ is the unique solution of \cref{main_pro}, $c_x = 1-\frac{\alpha\beta\os^2(B)}{1-\beta\ove{P}}$, $c_y = \frac{\alpha}{\beta}$, and
  $$\delta = 1 - \min\lt\{\alpha\mu_x(1-\alpha L_x), \alpha\beta\mine{BB\T + \frac{1}{\alpha}P}\rt\} \in (0, 1).$$
\end{theorem}

\begin{proof}
  See \cref{appendix_PDPG_convergence}.
\end{proof}

\begin{remark}
  \cref{PDPG_convergence} establishes that the linear convergence rate of PDPG is $\delta$. We now discuss how to derive the iteration complexity of PDPG\footnote{For PDPG, the oracle complexities of $\mathcal{A}$ and $\mathcal{B}$ coincide with its iteration complexity.} to achieve a desired accuracy $\epsilon$ based on this convergence rate. It is straightforward to show that, to ensure $c_x\norm{x^K-x^*}^2 + c_y\norm{y^K-y^*}^2 \leq \epsilon$, the number of iterations must satisfy $K \geq \frac{1}{1-\delta}\log\pa{\frac{c_x\norm{x^0-x^*}^2 + c_y\norm{y^0-y^*}^2}{\epsilon}}$. Note that $\delta$ is function of $\alpha$ and $\beta$. To minimize $\frac{1}{1-\delta}$, we need to choose appropriate values for $\alpha$ and $\beta$ , which is equivalent to maximizing $\min\lt\{\alpha\mu_x(1-\alpha L_x), \alpha\beta\mine{BB\T + \frac{1}{\alpha}P}\rt\}$. Appantly we should choose $\beta = \frac{\mu_x}{\os^2(B)+\mu_x\ove{P}}$. However, selecting $\alpha$ is less straightforward, as its influence on $\alpha\beta\mine{BB\T + \frac{1}{\alpha}P}$ is not clear. Nevertheless, if we consider only $\alpha\mu_x(1-\alpha L_x)$, the optimal choice of $\alpha$ would be $\alpha = \frac{1}{2L_x}$. Under this reasoning, it is reasonable to set $\alpha = \frac{1}{2L_x}$ and $\beta = \frac{\mu_x}{\os^2(B)+\mu_x\ove{P}}$. This yields $\frac{1}{1-\delta} = 2\frac{L_x}{\mu_x}\frac{\os^2(B)+\mu_x\ove{P}}{\mine{BB\T + 2L_x P}}$. Recall that $P \geq 0$. By applying Weyl's inequality \cite{horn2012matrix}, we obtain $\mine{BB\T + 2L_x P} \geq \mine{BB\T + L_x P} + \mine{P} \geq \mine{BB\T + L_x P}$. Additionally, noting that $\bO{C_1 + C_2} = \bO{\max(C_1, C_2)}$, we can finally derive the oracle complexity of PDPG as shown in \cref{1229}.
\end{remark}

\section{Linear Convergence of iDAPG for the Dual-Strongly-Convex Case}
\cref{PDPG_convergence} provides an optimistic answer for the previous question. To fully address the question, we first offer some intuition behind the conditions under which existing algorithms achieve linear convergence.
Since the saddle point of $\cL$ exists, according to \cite[Lemma 36.2]{rockafellar1970convex}, we have
\eqa{
  &\min_{x \in \R^{d_x}}\max_{y \in \R^{d_y}} f_1(x) + f_2(x) + y\T Bx - g_1(y) -g_2(y) \\
  =& -\min_{y \in \R^{d_y}} \varPhi(y) = \varphi(y) + g_2(y), \label{2158}
}
where $\varphi(y) = g_1(y) + (f_1+f_2)^*(-B\T y)$. Then, we obtain the following lemma.
\begin{lemma} \label{2131}
  Assume that \cref{basic_assump} holds, then $\varphi$ is $\pa{L_y+\frac{\os^2(B)}{\mu_x}}$-smooth. Furthermore, it holds that
  \begin{enumerate}
    \item if $g_1$ is $\mu_y$-strongly convex, then $\varphi$ is $\mu_y$-strongly convex;
    \item if $f_2 = 0$ and $B$ has full row rank, then $\varphi$ is $\frac{\mins^2(B)}{L_x}$-strongly convex;
    \item if $f_2 = 0$ and $g_2 = 0$, and $g_1$ is a linear function, then $\varphi$ is $\frac{\us^2(B)}{L_x}$-strongly convex on $\Range{B}$;
    \item if \cref{main_pro_assumption} holds, then $\varphi$ is $\frac{\mine{BB\T + L_xP}}{L_x}$-strongly convex.
  \end{enumerate}
\end{lemma}

\begin{proof}
  See \cref{appendix_2131}.
\end{proof}

According to \cref{2131}, if \cref{basic_assump} holds and any of the conditions listed above is satisfied, then $\varphi$ or $\varPhi$ is smooth and strongly convex (or restrictedly strongly convex).

For the unconstrained optimization problem \cref{2158}, it is well known that the classical proximal gradient descent method achieves linear convergence when $\varphi$ is smooth and strongly convex. Moreover, a faster convergence rate can be obtained by using APG \cite[Lecture 7]{vandenberghe2022om}. However, it is essential to consider the computational cost of calculating $\nabla \varphi(y)$:
\eqe{
  \nabla \varphi(y) = \nabla g_1(y) - Bx^*(y),
}
where
\eqe{ \label{2346}
  x^*(y) = \arg\min_{x \in \R^{d_x}} f_1(x) + f_2(x) + y\T Bx.
}
Since $f_1$ is smooth and strongly convex, \cref{2346} can also be solved by APG with linear convergence. However, utilizing the exact $x^*(y)$ requires solving \cref{2346} precisely, which is often impractical or even impossible for a general $f_1$. This challenge can be addressed by employing the inexact APG \cite{schmidt2011convergence} to solve \cref{2158}. The inexact APG relies only on an approximate gradient of $\varphi$ (implying that an approximate solution to \cref{2346} suffices) and can achieve the same convergence rates with the exact APG, provided the inexactness of the gradient is well-controlled.

Building on this approach, i.e., using the inexact APG to solve \cref{2158}, we propose iDAPG (\cref{alg:iDAPG}), which converges linearly when \cref{basic_assump} holds and $\varphi$ is strongly convex. According to \cref{2131}, the condition that $\varphi$ is strongly convex is weaker than all the conditions on which existing algorithms and PDPG achieve linear convergence.

Assume that $x\+$ satisfies the following error condition:
\eqe{ \label{2229}
\norm{x\+-x^*(z^k)}^2 \leq \frac{{\varepsilon^2_{k+1}}}{\os^2(B)}.
}
In the following theorem, we demonstrate that iDAPG achieves linear convergence if $\varepsilon_{k+1}^2$ decreases linearly.

\begin{theorem} \label{iDAPG_convergence}
  Assume that \cref{basic_assump} holds and $\varphi$ is $\mu_{\varphi}$-strongly convex, $x^k$ meets the condition \cref{2229}, and
  \eqe{ \label{1535}
    \varepsilon_{k+1}^2 = \theta \varepsilon_k^2,
  }
  where $\theta \in \pa{0, 1}$, then $x\+$ and $y^k$ generated by iDAPG satisfy that $\norm{x\+-x^*}^2$ and $\norm{y^k-y^*}^2$ converge as
  \begin{enumerate}
    \item $\bO{\max\pa{1-\frac{1}{\sqrt{\kappa_{\varphi}}}, \theta}^k}$ if $\theta \neq 1-\frac{1}{\sqrt{\kappa_{\varphi}}}$;
    \item $\bO{k^2\pa{1-\frac{1}{\sqrt{\kappa_{\varphi}}}}^k}$ if $\theta = 1-\frac{1}{\sqrt{\kappa_{\varphi}}}$;
  \end{enumerate}
  where $\kappa_{\varphi} = \frac{\os^2(B) + \mu_xL_y}{\mu_x\mu_{\varphi}}$.
\end{theorem}
\begin{proof}
  See \cref{appendix_iDAPG_convergence}.
\end{proof}

\begin{remark}
  Since $x^*(z^k)$ is not available before solving \cref{1356}, \cref{2229} cannot be directly applied in practice. By the definition of $x^*(z^k)$, we have
  \eqe{
    \0 \in \partial_x \cL(x^*(z^k), z^k) = \nabla f_1\pa{x^*(z^k)}+B\T z^k + \partial f_2\pa{x^*(z^k)}.
  }
  Using the $\mu_x$-strong convexity of $\cL$ w.r.t. $x$, we can obtain
  \eqe{ \label{1015}
  \norm{x\+-x^*(z^k)} &\leq \min_{S_{f_2}(x\+) \in \partial f_2(x\+)} \frac{1}{\mu_x}\norm{\nabla f_1(x\+)+S_{f_2}(x\+)+B\T z^k} \\
  &= \frac{1}{\mu_x}\dist{\0, \partial_x \cL(x\+, z^k)}.
  }
  Hence, we can instead use
  \eqe{
    \dist{\0, \partial_x \cL(x\+, z^k)} \leq \frac{{\mu_x\varepsilon_{k+1}}}{\os(B)}
  }
  to guarantee \cref{2229}, where $\dist{\0, \partial_x \cL(x\+, z^k)}$ can be easily calculated under the assumption that $f_2$ is proximal-friendly.
\end{remark}

Let us now analyze the outer and inner iteration complexities of iDAPG.
\begin{theorem} \label{iDAPG_complexity}
  Under the same assumptions and conditions with \cref{iDAPG_convergence}, choose a constant $c > 1$ and set
  \eqe{ \label{epsilon_1}
    \theta &= 1-\frac{1}{c\sqrt{\kappa_{\varphi}}}, \\
    \varepsilon_1 &= \pa{\sqrt{\theta}-\sqrt{1-\frac{1}{\sqrt{\kappa_{\varphi}}}}}\sqrt{\mu_{\varphi}\pa{\varPhi(y^0) - \varPhi(y^*)}}.
  }
  Then, the outer iteration complexity of iDAPG (to guarantee $\norm{x\+-x^*}^2 \leq \epsilon$ and $\norm{y^k-x^*}^2 \leq \frac{\mu^2_x}{\os^2(B)}\epsilon$) is given by
  \eqe{ \label{1456}
    \mO\lt(\sqrt{\kappa_{\varphi}} \log \lt(\frac{\mathcal{C}\pa{\varPhi(y^0) - \varPhi(y^*)}}{\epsilon}\rt)\rt),
  }
  where $\kappa_{\varphi}$ is defined in \cref{iDAPG_convergence} and $\mathcal{C} > 0$ is a constant.
  Moreover, if APG is employed to solve the subproblem \cref{1356}, with $x^k$ as the initial solution of APG at the $k$-th iteration, the inner iteration complexity of iDAPG is given by
  \eqe{ \label{1457}
    \tilde{\mO}\pa{\sqrt{\kappa_{\varphi}\kappa_x} \log \lt(\frac{\mathcal{C}\pa{\varPhi(y^0) - \varPhi(y^*)}}{\epsilon}\rt)},
  }
  where $\tilde{\mO}$ hides a logarithmic factor dependent on $\mu_x$, $\mu_{\varphi}$, $L_x$, $L_y$, $\os^2(B)$, and $c$.
\end{theorem}

\begin{proof}
  See \cref{appendix_iDAPG_complexity}.
\end{proof}

\begin{remark}
  In \cref{iDAPG_complexity}, the choice of $\varepsilon_1$ is made solely to achieve a tighter logarithmic constant in \cref{1456}. However, this setting is impractical since $\varPhi(y^*)$ is unknown prior to solving the problem.   An alternative choice is to define $\varepsilon_1$ as an upper bound of $\pa{\sqrt{\theta}-\sqrt{1-\frac{1}{\sqrt{\kappa_{\varphi}}}}}\sqrt{\mu_{\varphi}\pa{\varPhi(y^0) - \varPhi(y^*)}}$. Given that $\varPhi$ is $\mu_{\varphi}$-strongly convex, we have
  \eqe{
  &\varPhi(y^0) - \varPhi(y^*) \\
  \leq& \min_{S_{g_2}(y^0) \in \partial g_2(y^0)} \frac{1}{2\mu_{\varphi}}\norm{\nabla \varphi(y^0)+S_{g_2}(y^0)}^2 \\
  \leq& \frac{1}{2\mu_{\varphi}}\pa{ \min_{S_{g_2}(y^0) \in \partial g_2(y^0)} \pa{\norm{\nabla g_1(y^0)+S_{g_2}(y^0)-B\tx(y^0)}^2} + \os^2(B)\norm{\tx(y^0)-x^*(y^0)}^2} \\
  \overset{\cref{1015}}\leq& \frac{1}{2\mu_{\varphi}} \pa{\distsq{\0, \partial_y \cL(\tx(y^0), y^0)} + \frac{\os^2(B)}{\mu_x^2}\distsq{\0, \partial_x \cL(\tx(y^0), y^0)}} = C,
  \nonumber
  }
  where $\tx(y^0)$ is any approximate solution of \cref{2346} with $y = y^0$. Given that $f_2$ and $g_2$ are both proximal-friendly, $C$ can be easily computed. Consequently, in practice, we can set $\varepsilon_1 = \pa{\sqrt{\theta}-\sqrt{1-\frac{1}{\sqrt{\kappa_{\varphi}}}}}\sqrt{\mu_{\varphi}C}$.
\end{remark}

\section{Discussions}
When employing first-order algorithms to solve \cref{main_pro}, the primary computational cost stems from evaluating $\nabla f_1$, $\prox_{f_2}$, $\nabla g_1$, and $\prox_{g_2}$, as well as performing matrix-vector multiplications involving $B$ and $B\T$. Let $\mathcal{A}$ represent the evaluation of $\nabla f_1$ and $\prox_{f_2}$, and let $\mathcal{B}$ denote the evaluation of $\nabla g_1$ and $\prox_{g_2}$, and matrix-vector multiplications involving $B$ and $B\T$. To select an appropriate algorithm for solving \cref{main_pro}, it is essential to consider the oracle complexities of $\mathcal{A}$ and $\mathcal{B}$ across different algorithms.

For those primal-dual algorithms such as LPD, ABPD-PGS, APDG and PDPG, the oracle complexities of $\mathcal{A}$ and $\mathcal{B}$ coincide with their iteration complexities. However, for iDAPG, the oracle complexity of $\mathcal{A}$ corresponds to its inner iteration complexity, while the oracle complexity of $\mathcal{B}$ aligns with its outer iteration complexity. Based on \cref{2131,iDAPG_complexity}, we can readily derive the oracle complexities of iDAPG across different cases provided in \cref{1229}. From \cref{1229}, the following observations can be made:
\begin{enumerate}
  \item For the strongly-convex-concave case, ABPD-PGS consistently emerges as the optimal choice, as it is the only algorithm whose oracle complexity matches the theoretical lower bound.
  \item For the strongly-convex-strongly-concave case, iDAPG achieves a lower oracle complexity of $\mathcal{B}$ but a higher oracle complexity of $\mathcal{A}$ compared to other algorithms. Additionally, since the solution of \cref{main_pro} exists, \cref{main_pro} is equivalent to
        \eqe{ \label{2324}
          \min_{y \in \R^{d_y}}\max_{x \in \R^{d_x}} g_1(y) + g_2(y) - x\T B\T y - f_1(x) - f_2(x).
        }
        Clearly, if \cref{main_pro} falls under the strongly-convex-strongly-concave case, so does \cref{2324}. Consequently, if iDAPG is applied to solve \cref{2324}, it achieves a lower oracle complexity of $\mathcal{A}$ but a higher oracle complexity of $\mathcal{B}$ compared to other algorithms.
        Therefore, iDAPG is a preferable choice when the computational cost of $\mathcal{A}$ is significantly higher or lower than that of $\mathcal{B}$; otherwise, LPD, ABPD-PGS, or APDG may be a more suitable choice.
  \item For the strongly-convex-full-rank case, iDAPG establishes a lower oracle complexity of $\mathcal{B}$ but a higher oracle complexity of $\mathcal{A}$ compared to APDG. Thus, iDAPG should be selected when $g_2 \neq 0$ or when the computational cost of $\mathcal{B}$ is significantly higher than that of $\mathcal{A}$; otherwise, APDG is the more suitable choice.
  \item For the strongly-convex-linear case, the algorithm introduced in \cite{salim2022optimal} consistently stands out as the optimal choice, as it is the only method whose oracle complexity achieves the theoretical lower bound.
  \item For the case that \cref{basic_assump,main_pro_assumption} hold, iDAPG establishes a lower oracle complexity of $\mathcal{B}$ compared to PDPG. However, one might observe that the oracle complexity of $\mathcal{A}$ for PDPG appears lower than that of iDAPG when $L_y \geq \frac{\os^2(B)}{\mu_x}$ and $\mu_xL_y < \mine{BB\T + L_xP}$. In reality, this scenario is impossible because $\mine{BB\T + L_xP} \leq \mine(L_xP) + \ove{BB\T} = \os^2(B)$ (by Weyl's inequality and $P \geq 0$), which implies that the oracle complexity of $\mathcal{A}$ for iDAPG is never higher than that of PDPG. Consequently, iDAPG should be preferred over PDPG.
\end{enumerate}
Based on the above analysis, the most suitable algorithm for solving \cref{main_pro} can be selected to minimize computational costs under specific conditions.  

\begin{remark}
  As mentioned earlier, the most significant advantage of iDAPG is its ability to achieve linear convergence under a weaker condition than existing methods. Additionally, as shown in \cref{1229}, iDAPG exhibits lower oracle complexity of $\mathcal{B}$ but higher complexity of $\mathcal{A}$ compared to SOTA algorithms in some cases. This is particularly beneficial when the evaluation of $\mathcal{B}$ is significantly more expensive than that of $\mathcal{A}$. In such cases, iDAPG can be a more suitable choice for solving \cref{main_pro}.
\end{remark}


\bibliographystyle{IEEEtran}
\bibliography{../public/bib}

\onecolumn
\appendices \label{appendix}
\crefalias{section}{appendix}

\section{Proof of \cref{PDPG_convergence}} \label{appendix_PDPG_convergence}
We first prove that \cref{main_pro} has a unique solution under \cref{basic_assump,main_pro_assumption}.
\begin{lemma} \label{2103}
  Assume that \cref{basic_assump,main_pro_assumption} holds, and $g_3 = 0$, then \cref{main_pro} has a unique solution $(x^*, y^*)$.
\end{lemma}

\begin{proof}
  Let $(x^*, y^*)$ be a solution of \cref{main_pro}, given the definition of $g_1$, we have
  \eqe{ \label{2051}
    \nabla f_1(x^*) + B\T y^* &= 0, \\
    Bx^* - Py^* - S_{g_2}(y^*) &= 0,
  }
  where $S_{g_2}(y^*) \in \partial g_2(y^*)$ is a subgradient of $g_2$ at $y^*$.
  Since $f$ is strongly convex, $x^*$ must be unique, but $y^*$ may not be. Assume that there is different solution $(x^*, y^o)$, which satisfies $y^o \neq y^*$. We then have
  \eqe{ \label{2052}
    \nabla f_1(x^*) + B\T y^o &= 0, \\
    Bx^* - Py^o - S_{g_2}(y^o) &= 0,
  }
  where $S_{g_2}(y^o) \in \partial g_2(y^o)$.
  Combining \cref{2051,2052} gives
  \eqe{ \label{1708}
    B\T(y^*-y^o) &= 0, \\
    P(y^*-y^o) &= -\pa{S_{g_2}(y^*) - S_{g_2}(y^o)}.
  }
  Using \cref{1708} and the fact that $P \geq 0$ and $g_2$ is convex, we have
  \eqe{
    (y^*-y^o)\T P(y^*-y^o) &\geq 0, \\
    (y^*-y^o)\T P(y^*-y^o) &= -\dotprod{S_{g_2}(y^*) - S_{g_2}(y^o), y^*-y^o} \leq 0,
  }
  hence $(y^*-y^o)\T P(y^*-y^o) = 0$. Using \cref{1708} again, we can obtain
  \eqe{
    (y^*-y^o)\T(P + BB\T)(y^*-y^o) = 0.
  }
  According to \cref{main_pro_assumption}, $P + BB\T > 0$, which implies that $y^*=y^o$. A contradiction arises, hence $y^*$ must be unique.
\end{proof}

\begin{appendix_proof}[\cref{PDPG_convergence}]
  Let $\tx^k = x^k - x^*$ and $\ty^k = y^k - x^*$, we can obtain the error system of PDPG:
  \eqe{ \label{2259}
    \tx\+ &= \tx^k - \alpha\lt(\lt(\nabla f_1(x^k)-\nabla f_1(x^*)\rt) + B\T \ty^k\rt), \\
    \tz\+ &= \ty^k - \beta\lt(P\ty^k - B\tx\+\rt), \\
    \ty\+ &= \prox_{\beta g_2}(z\+)-\prox_{\beta g_2}(z^*),
  }
  which holds due to the optimality condition of \cref{main_pro} and the particular form of $g_1$.

  According to \cref{2259}, we have
  \eqe{ \label{1439}
    \norm{\tx\+}^2 =& \norm{\tx^k - \alpha\lt(\nabla f_1(x^k)-\nabla f_1(x^*)\rt)}^2 + \alpha^2\norm{B\T \ty^k}^2 \\
    &- 2\alpha\dotprod{\tx^k - \alpha\lt(\nabla f_1(x^k)-\nabla f_1(x^*)\rt), B\T \ty^k}
  }
  and
  \eqe{ \label{1440}
    \norm{\tz\+}^2 =& \norm{\ty^k}^2 + \beta^2\norm{P\ty^k - B\tx\+}^2 - 2\beta\norm{\ty^k}^2_P + 2\beta\dotprod{\ty^k, B\tx\+} \\
    =& \norm{\ty^k}^2 + \beta^2\norm{P\ty^k - B\tx\+}^2 - 2\beta\norm{\ty^k}^2_P - 2\alpha\beta\norm{B\T\ty^k}^2 \\
    &+ 2\beta\dotprod{\ty^k, B\lt(\tx^k - \alpha\lt(\nabla f_1(x^k)-\nabla f_1(x^*)\rt)\rt)}.
  }
  According to \cref{1428}, we immediately have $\beta < \frac{1}{\ove{P}}$, which guarantees that $0 < \beta\ove{P} < 1$, then applying Jensen's inequality to $\norm{P\ty^k - B\tx\+}^2$ gives that
  \eqe{ \label{1441}
    \norm{P\ty^k - B\tx\+}^2 &\leq \frac{1}{\beta\ove{P}}\norm{P\ty^k}^2 + \frac{1}{1-\beta\ove{P}}\norm{B\tx\+}^2 \\
    &\leq \frac{1}{\beta}\norm{\ty^k}^2_P + \frac{1}{1-\beta\ove{P}}\norm{B\tx\+}^2.
  }
  Combining \cref{1439,1440,1441}, we can obtain
  \eqe{ \label{1702}
    &\lt(1-\frac{\alpha\beta\os^2(B)}{1-\beta\ove{P}}\rt)\norm{\tx\+}^2 + \frac{\alpha}{\beta}\norm{\tz\+}^2 \\
    \leq& \norm{\tx^k - \alpha\lt(\nabla f_1(x^k)-\nabla f_1(x^*)\rt)}^2 + \frac{\alpha}{\beta}\norm{\ty^k}^2 - \alpha\norm{\ty^k}^2_P- \alpha^2\norm{B\T \ty^k}^2 \\
    \leq& (1-\alpha\mu_x(2-\alpha L_x))\norm{\tx^k}^2 + \frac{\alpha}{\beta}\lt(\norm{\ty^k}^2 - \alpha\beta\norm{\ty^k}^2_{BB\T + \frac{1}{\alpha}P}\rt),
  }
  where the last inequality follows from
  \eqe{
    &\norm{\tx^k - \alpha\lt(\nabla f_1(x^k)-\nabla f_1(x^*)\rt)}^2 \\
    =& \norm{\tx^k}^2 + \alpha^2\norm{\nabla f_1(x^k)-\nabla f_1(x^*)}^2 - 2\alpha\dotprod{\tx^k, \nabla f_1(x^k)-\nabla f_1(x^*)} \\
    \leq& \norm{\tx^k}^2 - \alpha(2-\alpha L_x)\dotprod{\tx^k, \nabla f_1(x^k)-\nabla f_1(x^*)} \\
    \leq& (1-\alpha\mu_x(2-\alpha L_x))\norm{\tx^k}^2,
  }
  where we use the strong convexity and smoothness of $f$. According to \cref{1428}, we also have
  \eqe{ \label{1638}
    \mu_x &\geq \frac{\beta\os^2(B)}{1-\beta\ove{P}}, \\
    \alpha\beta &< \frac{\mu_x}{L_x}\frac{1-\beta\ove{P}}{\os^2(B)} \leq \frac{1-\beta\ove{P}}{\os^2(B)}.
  }
  Let $c_x = 1-\frac{\alpha\beta\os^2(B)}{1-\beta\ove{P}}$, $c_y = \frac{\alpha}{\beta}$ and $\delta_x = 1-\alpha\mu_x(1-\alpha L_x)$, we can easily verify that $c_x, c_y > 0$ and $\delta_x \in (0, 1)$ based on \cref{1428,1638}. Then we have
  \eqe{ \label{1703}
    &(1-\alpha\mu_x(2-\alpha L_x))\norm{\tx^k}^2 \\
    =& \delta_x\norm{\tx^k}^2 - \alpha\mu_x\norm{\tx^k}^2 \\
    =& \delta_xc_x\norm{\tx^k}^2 - \alpha\lt(\mu_x - \delta_x\frac{\beta\os^2(B)}{1-\beta\ove{P}}\rt)\norm{\tx^k}^2 \\
    \leq& \delta_xc_x\norm{\tx^k}^2,
  }
  where the inequality follows from \cref{1638}. Also note that $BB\T + \frac{1}{\alpha}P > 0$, it follows that
  \eqe{ \label{1704}
  \norm{\ty^k}^2 - \alpha\beta\norm{\ty^k}^2_{BB\T + \frac{1}{\alpha}P} \leq \lt(1-\alpha\beta\mine{BB\T + \frac{1}{\alpha}P}\rt)\norm{\ty^k}^2.
  }
  Using Weyl's inequality gives
  \eqe{ \label{1700}
    \mine{BB\T + \frac{1}{\alpha}P} \leq \mine{\frac{1}{\alpha}P} + \ove{BB\T} = \os^2(B),
  }
  where the equality holds since $\mine{\frac{1}{\alpha}P} = 0$. Let $\delta_y = 1-\alpha\beta\mine{BB\T + \frac{1}{\alpha}P}$, we immediately know $\delta_y \in (0, 1)$ from \cref{1638,1700}. Combining \cref{1702,1703,1704}, we can obtain
  \eqe{
    c_x\norm{\tx\+}^2 + c_y\norm{\tz\+}^2 \leq \delta\pa{c_x\norm{\tx^k}^2 + c_y\norm{\ty^k}^2},
  }
  where $\delta = \max\{\delta_x, \delta_y\} \in (0, 1)$. Finally, using
  $$\norm{\ty\+}^2 = \norm{\prox_{\beta g_2}(z\+)-\prox_{\beta g_2}(z^*)} \leq \norm{\tz\+}^2$$
  completes the proof.
\end{appendix_proof}

\section{Proof of \cref{2131}} \label{appendix_2131}
\begin{proof}
  To complete the proof, we will use the following two lemmas, which reveal the duality between the strong convexity of a function $f$ and the smoothness of its Fenchel conjugate $f^*$.
\begin{lemma} \cite[Lecture 5]{vandenberghe2022om} \label{conjugate_smooth}
  Assume that $f: \R^n \rightarrow \exs$ is closed proper and $\mu$-strongly convex with $\mu>0$, then (1) $\dom{f^*} = \R^n$; (2) $f^*$ is differentiable on $\R^n$ with $\nabla f^*(y) = \arg\max_{x \in \dom{f}} y\T x - f(x)$; (3) $f^*$ is convex and $\frac{1}{\mu}$-smooth.
\end{lemma}

\begin{lemma} \cite[Theorem E.4.2.2]{hiriart2004fundamentals} \label{conjugate_strong_convex}
  Assume that $f: \R^n \rightarrow \R$ is convex and $L$-smooth with $L>0$, then $f^*$ is $\frac{1}{L}$-strongly convex on every convex set $\mathcal{Y} \subseteq \dom{\partial f^*}$, where $\dom{\partial f^*} = \set{y \in \R^n | \partial f^*(y) \neq \emptyset}$.
\end{lemma}

Let $\phi(y) = (f_1+f_2)^*(-B\T y)$. According to \cref{basic_assump}, $f_1+f_2$ is closed proper and $\mu_x$-strongly convex. By \cref{conjugate_smooth}, we conclude that $(f_1+f_2)^*$ is $\frac{1}{\mu_x}$-smooth. Consequently, $\phi$ is differentiable everywhere. Utilizing the smoothness of $f_1+f_2$, we have
\eqe{
  &\dotprod{\nabla \phi(y)-\nabla \phi(y'), y-y'} \\
  =& \dotprod{\nabla (f_1 + f_2)^*(-B\T y)-\nabla (f_1 + f_2)^*(-B\T y'), -B\T(y-y')} \\
  \leq& \frac{1}{\mu_x}\norm{B\T(y-y')}^2 \\
  \leq& \frac{\os^2(B)}{\mu_x}\norm{y-y'}^2, \ \forall y, y' \in \R^p.
}
Thus, $\phi$ is $\frac{\os^2(B)}{\mu_x}$-smooth. Since $g_1$ is $L_y$-smooth, it follows tha $\varphi$ is $\pa{L_y+\frac{\os^2(B)}{\mu_x}}$-smooth.

We will now prove the strong convexity of $\varphi$ for the cases mentioned above. The first case is straightforward. In the second case, since $f_1+f_2$ is $\mu_x$-strongly convex and $L_x$-smooth, it follows from \cref{conjugate_smooth,conjugate_strong_convex} that $(f_1+f_2)^*$ is $\frac{1}{L_x}$-strongly convex on $\R^p$. Using the strong convexity of $(f_1+f_2)^*$ yields
\eqe{
  &\dotprod{\nabla \phi(y)-\nabla \phi(y'), y-y'} \\
  =& \dotprod{\nabla (f_1+f_2)^*(-B\T y)-\nabla (f_1+f_2)^*(-B\T y'), -B\T(y-y')} \\
  \geq& \frac{1}{L_x}\norm{B\T(y-y')}^2 \\
  \geq& \frac{\mins^2(B)}{L_x}\norm{y-y'}^2, \ \forall y, y' \in \R^p,
}
where $\frac{\mins^2(B)}{L_x} > 0$ since $B$ has full row rank. Therefore, $\phi$ is $\frac{\mins^2(B)}{L_x}$-strongly convex, and so is $\varphi$.

For the third case, we again have $(f_1+f_2)^*$ being $\frac{1}{L_x}$-strongly convex on $\R^p$. Similarly, we obtain
\eqe{
  &\dotprod{\nabla \phi(y)-\nabla \phi(y'), y-y'} \\
  \geq& \frac{1}{L_x}\norm{B\T(y-y')}^2 \\
  \geq& \frac{\us^2(B)}{L_x}\norm{y-y'}^2, \ \forall y, y' \in \Range{B},
}
hence $\phi$ is $\frac{\us^2(B)}{L_x}$-strongly convex on $\Range{B}$, and so is $\varphi$.

For the last case, by \cref{main_pro_assumption}, we have $\mine{BB\T + L_xP} > 0$. It follows that
\eqe{
  &\dotprod{\nabla \varphi(y)-\nabla \varphi(y'), y-y'} \\
  =& \dotprod{\nabla (f_1+f_2)^*(-B\T y)-\nabla (f_1+f_2)^*(-B\T y'), -B\T(y-y')} + (y-y')\T P(y-y') \\
  &+ \dotprod{\nabla g_3(y) - \nabla g_3(y'), y-y'} \\
  \geq& \frac{1}{L_x}(y-y')\T(BB\T + L_xP)(y-y') \\
  \geq& \frac{\mine{BB\T + L_xP}}{L_x}\norm{y-y'}^2, \ \forall y, y' \in \R^p.
}
Hence, $\varphi$ is $\frac{\mine{BB\T + L_xP}}{L_x}$-strongly convex.
\end{proof}

\section{Proof of \cref{iDAPG_convergence}} \label{appendix_iDAPG_convergence}
\begin{lemma} \label{2236}
  Under the same assumptions and conditions with \cref{iDAPG_convergence}, $y^k$ generated by iDAPG satisfies
  \eqe{
    \varPhi(y^k) - \varPhi(y^*) \leq \lt(1-\frac{1}{\sqrt{\kappa_{\varphi}}}\rt)^k\lt(\sqrt{2(\varPhi(y^0) - \varPhi(y^*))} + \sqrt{\frac{2}{\mu_{\varphi}}}\mathcal{E}_k\rt)^2, \ \forall k \geq 1,
  }
  where $\mathcal{E}_k = \sum_{i=1}^k\lt(1-\frac{1}{\sqrt{\kappa_{\varphi}}}\rt)^{-i/2}\epsilon_i$.
\end{lemma}

\begin{proof}
  Given \cref{basic_assump}, $(f_1+f_2)^*(-B\T y)$ is $\frac{\os^2(B)}{\mu_x}$-smooth, then $\varphi$ is $\pa{\frac{\os^2(B)}{\mu_x} + L_y}$-smooth. According to \cref{2229}, we have
  \eqe{
    \norm{B\pa{x\+-x^*(z^k)}} \leq \varepsilon_{k+1}.
  }
  Therefore, iDAPG can be interpreted as an inexact APG applied to $\varPhi$, allowing us to utilize \cite[Proposition 4]{schmidt2011convergence} to complete the proof.
\end{proof}

\begin{lemma} \label{1432}
  Under the same assumptions and conditions with \cref{iDAPG_convergence}, $x\+$ generated by iDAPG satisfies
  \eqe{
  \norm{x\+ - x^*}^2 \leq \frac{2\os^2(B)}{\mu^2_x}\lt(\lt(1+\frac{\sqrt{\kappa_{\varphi}}-1}{\sqrt{\kappa_{\varphi}}+1}\rt)\norm{y^k-y^*} + \frac{\sqrt{\kappa_{\varphi}}-1}{\sqrt{\kappa_{\varphi}}+1}\norm{y^{k-1}-y^*}\rt)^2 + \frac{{2\varepsilon^2_{k+1}}}{\os^2(B)}, \ \forall k \geq 1.
  }
\end{lemma}
\begin{proof}
  According to the definitions of $x^*$ and $x^*(z^k)$, we have
  \eqe{
    -B\T y^* \in \nabla f_1(x^*) + \partial f_2(x^*), \\
    -B\T z^k \in \nabla f_1(x^*(z^k)) + \partial f_2(x^*(z^k)).
  }
  Note that $f_1+f_2$ is $\mu_x$-strongly convex, using \cite[Lemma 3]{zhou2018fenchel} gives
  \eqe{ \label{1501}
    \norm{x^*(z^k)-x^*} \leq& \frac{1}{\mu_x}\norm{B\T\pa{z^k-y^*}} \\
    \leq& \frac{\os(B)}{\mu_x}\norm{z^k-y^*} \\
    \overset{\text{iDAPG}}\leq& \frac{\os(B)}{\mu_x}\pa{\pa{1+\frac{\sqrt{\kappa_{\varphi}}-1}{\sqrt{\kappa_{\varphi}}+1}}\norm{y^k-y^*} + \frac{\sqrt{\kappa_{\varphi}}-1}{\sqrt{\kappa_{\varphi}}+1}\norm{y^{k-1}-y^*}}.
  }
  Combining the above inequality with
  \eqe{
    \norm{x\+ - x^*}^2 \leq 2\norm{x\+ - x^*(z^k)}^2 + 2\norm{x^*(z^k) - x^*}^2
  }
  completes the proof.
\end{proof}

\begin{appendix_proof}[\cref{iDAPG_convergence}]
  According to \cref{1535}, we have
  \eqe{ \label{1500}
    \mathcal{E}_k &= \sum_{i=1}^k\lt(1-\frac{1}{\sqrt{\kappa_{\varphi}}}\rt)^{-i/2}\epsilon_i = \frac{\varepsilon_1}{\sqrt{\theta}}\sum_{i=1}^k\lt(\frac{\theta}{1-\frac{1}{\sqrt{\kappa_{\varphi}}}}\rt)^{i/2} < \frac{\varepsilon_1}{\sqrt{1-\frac{1}{\sqrt{\kappa_{\varphi}}}}-\sqrt{\theta}}, \ \text{if} \ \theta < 1-\frac{1}{\sqrt{\kappa_{\varphi}}}, \\
    \mathcal{E}_k &= \frac{\pa{\lt(\frac{\theta}{1-\frac{1}{\sqrt{\kappa_{\varphi}}}}\rt)^{k/2}-1}\varepsilon_1}{\sqrt{\theta}-\sqrt{1-\frac{1}{\sqrt{\kappa_{\varphi}}}}} < \frac{\varepsilon_1}{\sqrt{\theta}-\sqrt{1-\frac{1}{\sqrt{\kappa_{\varphi}}}}}\lt(\frac{\theta}{1-\frac{1}{\sqrt{\kappa_{\varphi}}}}\rt)^{k/2}, \ \text{if} \ \theta > 1-\frac{1}{\sqrt{\kappa_{\varphi}}}, \\
    \mathcal{E}_k &= k\frac{\varepsilon_1}{\sqrt{\theta}}, \ \text{if} \ \theta = 1-\frac{1}{\sqrt{\kappa_{\varphi}}}.
  }
  Since $\varphi$ is $\mu_{\varphi}$-strongly convex, so is $\varPhi$, which implies that
  \eqe{ \label{2248}
    \norm{y^k-y^*}^2 \leq \frac{2}{\mu_{\varphi}}\pa{\varPhi(y^k) - \varPhi(y^*)}.
  }
  Then, we can complete the proof via combining \cref{2236,1500}, \cref{1432,2248}, and \cref{1535}.
\end{appendix_proof}

\section{Proof of \cref{iDAPG_complexity}} \label{appendix_iDAPG_complexity}
\begin{appendix_proof}[\cref{iDAPG_complexity}]
  Note that $\theta > 1-\frac{1}{\sqrt{\kappa_{\varphi}}}$. According to \cref{1501}, \cref{2236}, and \cref{1500}, we have
  \eqe{
    &\norm{x^*(z^k)-x^*}^2 \\
    <& C_1\lt(\sqrt{2(\varPhi(y^0) - \varPhi(y^*))} + \frac{\sqrt{\frac{2}{\mu_{\varphi}}}\varepsilon_1}{\sqrt{\theta}-\sqrt{1-\frac{1}{\sqrt{\kappa_{\varphi}}}}}\lt(\frac{\theta}{1-\frac{1}{\sqrt{\kappa_{\varphi}}}}\rt)^{k/2}\rt)^2\lt(1-\frac{1}{\sqrt{\kappa_{\varphi}}}\rt)^k \\
    \overset{\cref{epsilon_1}}\leq& C_2\varepsilon^2_1\theta^k,
  }
  where $C_1 = \frac{2\os^2(B)}{\mu^2_x\mu_{\varphi}}\lt(\frac{\sqrt{\kappa_{\varphi}}\pa{2+\sqrt{1-\frac{1}{\sqrt{\kappa_{\varphi}}}}}}{\sqrt{\kappa_{\varphi}}+1}\rt)^2$ and $C_2 = \frac{8C_1}{\mu_{\varphi}\pa{\sqrt{\theta}-\sqrt{1-\frac{1}{\sqrt{\kappa_{\varphi}}}}}^2}$.
  It follows that
  \eqe{ \label{2041}
    \norm{x\+ - x^*}^2 \leq& 2\norm{x^*(z^k)-x^*}^2 + \frac{{2\varepsilon^2_1\theta^k}}{\os^2(B)} \\
    &< C_3\varepsilon^2_1\theta^k,
  }
  where $C_3 = 2\pa{C_2+\frac{1}{\os^2(B)}}$. By \cref{epsilon_1,2041}, we immediately obtain \cref{1456}, where
  $$\mathcal{C}=\frac{32\os^2(B)\kappa_{\varphi}\pa{2+\sqrt{1-\frac{1}{\sqrt{\kappa_{\varphi}}}}}^2}{\mu^2_x\mu_{\varphi}\pa{\sqrt{\kappa_{\varphi}}+1}^2} + \frac{2\mu_{\varphi}}{\os^2(B)}\pa{\sqrt{\theta}-\sqrt{1-\frac{1}{\sqrt{\kappa_{\varphi}}}}}^2.$$

  We now proceed to prove \cref{1457}. Define $F_k(x) = f_1(x) + f_2(x) + \dotprod{B\T z^k, x}$. Consider the APG given in \cite[Lecture 7]{vandenberghe2022om}, with $\theta_0 = 1$ initialized \footnote{We should note that setting $\theta_0 = 1$ is not necessary if $f_2 = 0$. For the problem $\min_x f(x) = g(x) + h(x)$, where $g$ is $\mu$-strongly convex and $L$-smooth, we define $\kappa = \frac{L}{\mu}$. APG in \cite[Lecture 7]{vandenberghe2022om} satisfies $f(x^k) - f^* \leq \pa{1-\frac{1}{\sqrt{\kappa}}}^{k-1}\pa{(1-\theta_0)(f(x^0)-f^*) + \frac{\theta_0^2}{2t_0}\norm{x_0-x^*}^2}$ for $\theta_0 \in (0, 1]$. We set $\theta_0 = 1$ to eliminate $f(x^0)-f^*$ from the upper bound. However, when $f_2 = 0$ (i.e., $h = 0$), we can bound $f(x^0)-f^*$ using $\norm{x_0-x^*}^2$, making it unnecessary to set $\theta_0 = 1$.} and let $x^k$ serve as the initialized solution of APG at the $k$-th iteration. Since $f_1(x) + \dotprod{B\T z^k, x}$ is $\mu_x$-strongly convex and $L_x$-smooth, the number of iterations of APG required to guarantee
  \eqe{ \label{2058}
  F_k(x\+) - F_k(x^*(z^k)) \leq \frac{\mu_x}{2}\varepsilon^2_{k+1}
  }
  is bounded by
  \eqe{
  \sqrt{\kappa_x}\log \pa{\frac{\kappa_x\norm{x^k-x^*(z^k)}^2}{\varepsilon^2_{k+1}}} + 1.
  }
  Note that $F_k$ is also $\mu_x$-strongly convex, we then have
  \eqe{
    \norm{x\+ - x^*(z^k)}^2 \leq \frac{2}{\mu_x}\pa{F_k(x\+) - F_k(x^*(z^k))}.
  }
  Hence, \cref{2058} is sufficient to ensure $\norm{x\+ - x^*(z^k)}^2 \leq \varepsilon^2_{k+1}$.
  Note that
  \eqe{
    \norm{x^k - x^*(z^k)}^2 \leq& 2\norm{x^k - x^*}^2 + 2\norm{x^*(z^k)-x^*}^2 \\
    <& 2C_3\varepsilon^2_1\theta^{k-1} + 2C_2\varepsilon^2_1\theta^k \\
    \leq& 2\pa{\theta C_2+C_3}\varepsilon^2_k,
  }
  then we can obtain
  \eqe{
  \frac{\norm{x^k - x^*(z^k)}^2}{\varepsilon^2_{k+1}} \overset{\cref{1535}}< \frac{2\pa{\theta C_2+C_3}}{\theta}.
  }
  Therefore, the number of APG iterations (i.e., inner iterations) at the $k$-th iteration of iDAPG is bounded by $\sqrt{\kappa_x}\log \pa{\frac{2\pa{\theta C_2+C_3}\kappa_x}{\theta}} + 1$ for all $k \geq 1$. This completes the proof.
\end{appendix_proof}

\end{document}